\documentclass[12pt]{amsart}

\usepackage{amsmath, amssymb, amscd}

\usepackage{mathptmx}

\newcommand{\mm}{\mathfrak m}

\newcommand{\N}{\mathbb{N}}

\newcommand{\Z}{\mathbb{Z}}

\newcommand{\Bc}{\mathcal{B}}

\newcommand{\Sc}{\mathcal{S}}

\DeclareMathOperator{\pnt}{\raise 0.5mm \hbox{\large\bf.}}

\DeclareMathOperator{\gin}{gin}
\DeclareMathOperator{\spanK}{span}
\DeclareMathOperator{\rlex}{rlex}

\DeclareMathOperator{\ini}{in}

\DeclareMathOperator{\Kern}{Ker}
\DeclareMathOperator{\chara}{char}

\DeclareMathOperator{\rank}{rank}
\DeclareMathOperator{\reg}{reg}

\DeclareMathOperator{\Tor}{Tor}
\DeclareMathOperator{\Sh}{Sh}
\DeclareMathOperator{\maxSh}{maxSh}
\DeclareMathOperator{\minSh}{minSh}
\DeclareMathOperator{\GL}{GL}

\def\ie{\hbox{i.\,e. }}

\newtheorem{thm}{\bf Theorem}[section]
\newtheorem{lem}[thm]{\bf Lemma}

\newtheorem{prop}[thm]{\bf Proposition}

\theoremstyle{definition}

\newtheorem{rem}[thm]{\bf Remark}
\newtheorem{ex}[thm]{\bf Example}

\theoremstyle{plain}
\newtheorem*{thm*}{Theorem}

\title{Generic initial ideals and fibre products}

\author{Aldo Conca}
\address{Dipartimento di Matematica, Universit\'a di Genova, 16146 Genova, Italy}
\email{conca@dima.unige.it}

\author{Tim R\"omer}
\address{Fachbereich Mathematik/Informatik, Institut f\"ur Mathematik, Universit\"at Osna\-br\"uck, Albrechtstr. 28a, 49069 Osnabr\"uck, Germany}
\email{troemer@uos.de}

\begin{document}

\begin{abstract}
We study the behavior of generic initial ideals with respect to fibre products.
In our main result we determine the generic initial ideal of the fibre product with respect to
the reverse lexicographic order. As an application we compute
the symmetric algebraic shifted complex of two disjoint simplicial complexes as was conjectured by Kalai. This result is the symmetric analogue of a theorem of Nevo
who determined the exterior algebraic shifted complex of two disjoint simplicial complexes as predicted by Kalai.
 \end{abstract}

\thanks{Both authors thank CRUI and DAAD  for the financial support given through the Vigoni project}

\maketitle
%------------------------------------------------------------------------
%
%
%
%------------------------------------------------------------------------
\section{Introduction}
\label{intro}
Let $K$ be a field of characteristic $0$, $S=K[x_1,\dots,x_n]$ and $I$ be a graded  ideal of $S$.
Unless otherwise said, initial terms are taken with respect to the degree reverse lexicographic order,
rlex for short.  The generic initial ideal $\gin(I)$ of $I$ is an  important object.
It reflects many  homological, geometrical and combinatorial properties of the associated rings and varieties.
The ideal   $\gin(I)$ is defined as the ideal of  initial  terms of  $\varphi(I)$ where $\varphi$
is a generic element of $\GL_n(K)$. Here $\GL_n(K)$  acts  on $S$ as the group of
graded $K$-algebra isomorphisms.
See, e.g.,  Eisenbud \cite{EI95} for the construction,
in particular for the precise definition of ``generic'' element,
and \cite{GR98, HE01} for results related to the generic initial ideal.

 In practice, to compute $\gin(I)$,  one takes a random element $\varphi\in \GL_n(K)$ and computes, with  CoCoA \cite{CoCo}, Macaulay 2 \cite{MA2} or  Singular \cite{GPS05}, the ideal of leading terms of $\varphi(I)$. The resulting ideal is, with very high probability,  $\gin(I)$.
From the theoretical point of view  it is in general   difficult  to describe $\gin(I)$ explicitly. Just to give an example, important conjectures, such as the Fr\"oberg conjecture on  Hilbert functions,  would be solved if we  could only say what is the gin  of a generic complete intersection.

Generic initial ideals can  also be defined in the exterior algebra.
In the exterior algebra context  Kalai made in \cite{KA01} conjectures about the behavior of generic initial ideals  of squarefree monomial ideals  with respect standard operations  on simplicial complexes.
Two of these  operations have purely algebraic counterparts:   they corresponds   to the tensor product  and the fibre product over $K$ of standard graded $K$-algebras.  If $A=\oplus A_i$ and $B=\oplus B_i $ are standard graded $K$-algebras
the tensor product $A\otimes_K B=\oplus ( A_i\otimes_K B_j)$ is a (bi)graded $K$-algebra and the fibre product  $A\circ B$ is,  by definition, $A\otimes_K B/(A_1\otimes B_1)$.  Note that we have the short exact sequence
\begin{equation}\label{strfibre}
0\to A\circ B \to  A\oplus B\to K\to 0.
\end{equation}

If we have presentations $A=K[x_1,\dots,x_n]/I$ and $B=K[x_{n+1},\dots,x_{n+m}]/J$ then $A\otimes_K B$ has  presentation    $K[x_1,\dots,x_{n+m}]/I+J$  and $A\circ B$ has presentation $K[x_1,\dots,x_{n+m}]/F(I,J)$ where
$$F(I,J)=I+J+Q$$
 and $Q=(x_i : 1\leq i\leq n)(x_j : n+1\leq j\leq m)$.

Denote by  $\gin_{n}(I)$ the  gin  of   $I$ in $K[x_1,\dots,x_n]$ with respect to rlex induced by $x_1>\dots>x_n$.
Similarly denote by   $\gin_{m}(J)$ the  gin  of   $J$ in $K[x_{n+1},\dots,x_{n+m}]$  with respect to rlex
induced by $x_{n+1}>\dots>x_{n+m}$.
Kalai's conjectures \cite{KA01} correspond, in the symmetric algebra,  to the following statements:

 \begin{equation}\label{K1}
 \gin(I+J)=\gin(\gin_n(I)+\gin_m(J)),
  \end{equation}

\begin{equation}\label{K2}
\gin(F(I,J))=\gin(F(\gin_n(I),\gin_m(J)).
 \end{equation}

As observed by Nevo \cite{NE03},  the exterior algebra counterpart of
Equality (\ref{K1})  does not hold.
Hence it is not surprising to discover that (\ref{K1}) does not hold in the symmetric algebra as well.
Here is a simple example:

\begin{ex}
Set  $n=2, m=3$ and $I=(x_1,x_2)^2$ and
$J=(x_3^2,x_3x_4, x_4x_5)$. Then  $\gin_2(I)=I$ since $I$ is strongly stable,
$\gin_3(J)=(x_3,x_4)^2$
since $J$ defines a Cohen-Macaulay ring of codimension $2$.
But  $\gin(I+J)$ and $\gin(\gin_n(I)+\gin_m(J))$  differ  already in degree $2$
since $x_1x_4$ belongs to the second, but not to the first ideal.
\end{ex}

On the other hand,  Nevo \cite{NE03}  proved  the exterior algebra analog of (\ref{K2}) for squarefree monomial ideals.  Our main result  asserts that  (\ref{K2}) holds:
\begin{thm}\label{Nevo1} Let
$I \subset K[x_1,\dots,x_n]$ and
$J \subset K[x_{n+1},\dots,x_{n+m}]$
be graded ideals.
Then we have
$
\gin( F(I,J))
=
\gin( F(\gin_{n}(I),\gin_{m}(J))).$
\end{thm}

Using the results in Section \ref{fibre} it is also possible
to compute explicitly
$\gin( F(I,J))$ in some cases. As an example we show that
$$
\gin(Q)
=
(x_ix_j: i+j \leq n+m,\ i \leq j \text{ \rm and } i \leq \min\{n,m\})
$$
and further that $\gin(Q^k)=\gin(Q)^k$ for all $k\in \N$.
In Section \ref{symmetric} we present an application related to simplicial complexes. Let $\Gamma$ be a simplicial complex on the vertex set $\{x_1,\dots,x_n\}$.
We denote the {\em symmetric algebraic shifted complex} of $\Gamma$
by $\Delta^s(\Gamma)$.
In general a shifting operation $\Delta$ in the sense of Kalai \cite{KA01} replaces a complex $\Gamma$ with
a combinatorially simpler complex $\Delta (\Gamma)$ that still
detects some combinatorial properties of $\Gamma$.
Symmetric algebraic shifting $\Delta^s$ is an example of such an operation and
can be realized by using  a kind of
polarization of the generic initial ideal with respect
to the reverse lexicographic order
of the Stanley-Reisner ideal of $\Gamma$ in the polynomial ring $K[x_1,\dots,x_n]$.
See Section \ref{symmetric} for precise
definitions and more details.
For two simplicial complexes
$\Gamma_1$ on the vertex set $\{x_1,\dots,x_n\}$ and
$\Gamma_2$ on the vertex set $\{x_{n+1},\dots,x_{n+m}\}$
we denote by
$\Gamma_1 \uplus \Gamma_2$
the {\em disjoint union}
of $\Gamma_1$ and $\Gamma_2$ on the vertex set $\{x_1,\dots,x_n,x_{n+1},\dots,x_{n+m}\}$.
In Section \ref{symmetric} we prove:

\begin{thm}\label{Nevo2} We have
$
\Delta^s(\Gamma_1 \uplus \Gamma_2)
=
\Delta^s(\Delta^s(\Gamma_1) \uplus \Delta^s(\Gamma_2)).
$
\end{thm}
This result was conjectured by Kalai in \cite{KA01} and it was proved for exterior shifting by Nevo \cite{NE03}
who noted also that his approach could be used to prove the statement
for symmetric shifting.
Indeed, our methods and strategy of proofs as presented in Section \ref{linear} and
Section \ref{fibre} follow the ideas of Nevo and can be seen as a symmetric version of his results in the exterior algebra.

%------------------------------------------------------------------------
%
%
%
%------------------------------------------------------------------------

\section{Linear Algebra and Generic initial ideals}
\label{linear}

In this section we present some tools from linear algebra used in
this paper. We follow ideas of Nevo \cite{NE03} who treated the exterior case of the results which will be presented in this section.

Let $K$  be a field with $\chara K=0$ and $S=K[x_1,\dots,x_n]$ be the
standard graded polynomial ring.
Given a monomial
$x^a \in S_d$ we
define a $K$-linear map $\tau_{x^a} \colon S \to S$
by
$$
\tau_{x^a} (x^b)
=
\begin{cases}
x^{b-a} & \text{if } a \preceq b,\\
0 & \text{else}.
\end{cases}
$$
For a polynomial $g=\sum_{a\in \N^n} \lambda_a x^a \in S$
we extend this definition
and consider the $K$-linear map $\tau_{g}= \sum_{a \in \N^n}\lambda_a \tau_{x^a}$.
In the following we will use the rule
$$
\tau_{g h} (l)=\tau_g(\tau_h(l)) \text{ for homogeneous } g,h,l \in S
$$
which is easily verified.

Given a graded ideal $I \subset S$ and the reverse lexicographic order $<_{\rlex}$
induced by $x_1>\dots>x_n$ on $S$,
the algebra $S/I$ has a $K$-basis $\Bc(S/I)$
consisting of those monomial $x^a$ such that $x^a \not\in \ini(I)$
where $\ini(I)$ always denotes the initial ideal of $I$ with respect
to $\rlex$.
%Note that we always assume that $I$ is generated in degrees $\geq 2$.

Induced by the inclusion of $\Bc(S/I)$ into $\Bc(S)$,
we may consider $S/I$ as a graded $K$-vector subspace of $S$.
(So we are not using the quotient structure of $S/I$ as an $S$-module in the following
if not otherwise stated.)
Given a homogeneous element $g \in S_d$ the map
$\tau_g$ can be restricted to $S/I$. The main object we are studying
are the various kernels
$$
\Kern_{S/I} (\tau_g)_e =
\Kern(\tau_g\colon (S/I)_e \to (S/I)_{e-d}).
$$

We fix now a second generic $K$-basis
$f_j=\sum_{j=1}^n a_{ij}x_i$ of $S_1$ where $A=(a_{ij}) \in K^{n\times n}$ is invertible
and study the map
$\tau_{f^a}$ where $f^a$ is a monomial in the new $K$-basis $f_1,\dots,f_n$.
Generic means here that with the
automorphism
$\varphi\colon S \to S, x_i \mapsto f_i$
we can compute the generic initial ideal
$\gin(I)$
with respect to $\rlex$ induced by $x_1>\dots>x_n$
of a given graded ideal $I \subset S$
via
$\gin(I)=\ini(\varphi^{-1}(I))$.

We need the following alternative way to decide whether a monomial
belongs to a generic initial ideal or not. This criterion is well-known
to specialists but we include it for the sake of completeness.
For a (homogeneous) element $g \in S$
we denote by $\overline{g}$ the residue class in $S/I$
(induced by the quotient structure of $S/I$).

\begin{lem}
\label{computegin}
Let $I \subset S$ be a graded ideal and $x^a \in S$ be a monomial.
Then
$$
x^a \in \gin(I) \Longleftrightarrow
\overline{f^a} \in \spanK_K\{\overline{f^b} : x^b<_{\rlex} x^a,\ |b|=|a|   \} \subset S/I.
$$
\end{lem}
\begin{proof}
We have that
\begin{eqnarray*}
&&
x^a \in \gin(I)=\ini(\varphi^{-1}(I))\\
&\Leftrightarrow&
x^a-\sum_{x^b<_{\rlex} x^a,\ |b|=|a|} \lambda_b x^b \in  \varphi^{-1}(I)
\text{ for some } \lambda_b \in K\\
&\Leftrightarrow&
\overline{x^a} \in \spanK_K \{ \overline{x^b}: x^b <_{\rlex} x^a,\ |b|=|a| \}
\subset S/\varphi^{-1}(I)\\
&\Leftrightarrow&
\overline{f^a} \in \spanK_K \{\overline{f^b}: x^b<_{\rlex}x^a,\ |b|=|a|\} \subset S/I.
\end{eqnarray*}
This concludes the proof.
\end{proof}

Let $\Bc(S/I)=\{x^c: x^c \not\in \ini(I)\}$ be the $K$-basis
of
$S/I$ as discussed above. Similarly
$\Bc(S/\gin(I))$ is the natural
$K$-basis of $S/\gin(I)$, \ie those monomials $x^c$ with $x^c \not\in \gin(I)$.
We consider these monomials either as monomials in $S$
or in the corresponding residue class ring of $S$.
The generic initial ideal
$\gin(I)$
can also be characterized via the
kernels $\Kern_{S/I} (\tau_{f^a})_e$. More precisely, we have:

\begin{prop}
\label{kern1}
Let $d>0$ be a positive integer,
$x^a \in S_d$ be a monomial and $I \subset S$ be a graded ideal.
Then
$$
\dim_K \bigcap_{x^b<_{\rlex}x^a,\  |b|=d} \Kern_{S/I} (\tau_{f^b})_d=|\{x^c \in \Bc(S/\gin(I)) : |c|=d,\ x^a \leq_{\rlex} x^c\}|.
$$
In particular,
$x^a \in \Bc(S/\gin(I))$ if and only if
$$
\dim_K \bigcap_{x^b<_{\rlex}x^a,\  |b|=d} \Kern_{S/I} (\tau_{f^b})_d>\dim_K \bigcap_{x^b\leq_{\rlex} x^a,\  |b|=d} \Kern_{S/I} (\tau_{f^b})_d.
$$
\end{prop}
\begin{proof}
Let $g \in (S/I)_d$ be a homogenous element expressed as a linear combination of the elements of $\Bc(S/I)$,
\ie
$$
g=\sum_{x^c \in \Bc(S/I)}  \lambda_c x^c
\text{ for some } \lambda_c \in K.
$$
Write $f^b$ for $b \in \N^n$ with $|b|=d$ as
$$
f^b= \sum_{a' \in \N^n,\ |a'|=d} A_{ba'} x^{a'}
$$
where $A_{ba'}$ is a polynomial in the $a_{ij}$ of the base change matrix $A=(a_{ij})$.
We get that
$$
\tau_{f^b}(g)=\sum_{x^c \in \Bc(S/I)}  A_{bc}\lambda_c.
$$
Hence
$$
\dim_K \bigcap_{x^b<_{\rlex}x^a,\  |b|=d} \Kern_{S/I} (\tau_{f^b})_d
$$
equals the dimension of the
solution space $\Sc$
of the linear systems of equations
$M_a z=0$ with the matrix
$M_a=(A_{bc})$
where
$x^b<_{\rlex}x^a$, $x^c \in \Bc(S/I)$ and $|b|=|c|=d$.
It follows from Lemma \ref{computegin}
that
$$
\rank M_a =
|\{x^c \in \Bc(S/\gin(I)) : |c|=d,\ x^c<_{\rlex}x^a  \}|.
$$
We compute
\begin{eqnarray*}
\dim_K \bigcap_{x^b<_{\rlex}x^a,\  |b|=d} \Kern_{S/I} (\tau_{f^b})_d
&=&
\dim_K \Kern(M_a\colon(S/I)_d \to (S/I)_d)\\
&=&
\dim_K (S/I)_d - \rank M_a\\
&=&
\dim_K (S/\gin(I))_d - \rank M_a\\
&=&
\{x^c \in \Bc(S/\gin(I)) : |c|=d,\ x^a \leq_{\rlex} x^c  \}
\end{eqnarray*}
and this is what we wanted.
\end{proof}

For later applications we give a second
characterization  of the membership in the generic initial ideals
which is similar to Proposition \ref{kern1}.
At first we need the following Lemma:

\begin{lem}
\label{helperker}
Let $a \in \N^n$ and $d=|a|<e$.
Then
$$
\Kern_{S/I} (\tau_{f^a})_e =
\bigcap_{j=1}^n \Kern_{S/I} (\tau_{f^{a+\varepsilon_j}})_e
$$
\end{lem}
\begin{proof}
Let $g \in \Kern_{S/I} (\tau_{f^a})_e$.
Then for $i=1,\dots,n$ we have that
$
\tau_{f^{a+\varepsilon_j}}(g) = \tau_{f_j} (\tau_{f^{a}}(g)) =0.
$
Thus
$g \in
\cap_{i=1}^n \Kern_{S/I} (\tau_{f^{a+\varepsilon_j}})_e
$.

On the other hand let
$g \in
(S/I)_e
\setminus
\Kern_{S/I} (\tau_{f^a})_e$.
Then
$0\neq \tau_{f^a}(g)\in (S/I)_{e-d}$ is a nonzero element of degree $>0$.
Since the set $\{f^c: c \in \N^n,\ |c|=e-d\}$ is a $K$-basis for $S_{e-d}$,
there must exist an $f^c$ such that $\tau_{f^c}(\tau_{f^a}(g))\neq 0$.
Choose an $i \in \N$ with $c_i>0$. Then
$$
0\neq\tau_{f^c}(\tau_{f^a}(g))= \tau_{f^{a+c}}(g)=\tau_{f^{c-\epsilon_i}}(\tau_{f^{a+\epsilon_i}}(g))\neq 0
$$
and hence $\tau_{f^{a+\epsilon_i}}(g) \neq 0$.
Thus $g \notin \cap_{i=1}^n \Kern_{S/I} (\tau_{f^{a+\varepsilon_i}})_e$.
This concludes the proof.
\end{proof}

For a monomial $1\neq x^b \in S$
with $b \in \N^{n}$
we set
$\min(x^b)=\min(b)=\min\{i : b_i>0\}$.
Furthermore we define $\min(1)=1$.
Given a monomial $x^a \in S_d$ we denote
by
$$
\Sh(a) = \{x^b \in S_{d+1}:  x^b/x_{\min(b)}=x^a\}
$$
the {\em shadow} of $x^a$ by multiplying with $x_i$ such that $i \leq \min(a)$.
Let
$$
\maxSh(a)=\max_{<_{\rlex}} \Sh(a)
\text{ and }\minSh(a)=\min_{<_{\rlex}} \Sh(a)
$$
be the maximal and minimal element of $\Sh(a)$
with respect to $<_{\rlex}$ on $S$.
Now we partition the generic initial ideal using the set $\Sh(a)$.

\begin{prop}
\label{ker2}
Let $d> 0$, $x^a \in S_d$ be a monomial and $I \subset S$ be a graded ideal.
Then
\begin{eqnarray*}
&&|
\Sh(a) \cap \Bc(S/\gin(I))|\\
&=&
\dim_K \bigcap_{x^b<_{\rlex}x^a,\ |b|=d} \Kern_{S/I} (\tau_{f^b})_{d+1}-
\dim_K \bigcap_{x^b\leq_{\rlex} x^a,\ |b|=d} \Kern_{S/I} (\tau_{f^b})_{d+1}
\end{eqnarray*}
\end{prop}
\begin{proof}
It follows from Lemma \ref{helperker} that
$$
\dim_K \bigcap_{x^b<_{\rlex}x^a,\ |b|=d} \Kern_{S/I} (\tau_{f^b})_{d+1}=
\dim_K \bigcap_{x^b<_{\rlex}x^a,\ |b|=d}\bigcap_{j=1}^n \Kern_{S/I} (\tau_{f^{b+\epsilon_j}})_{d+1}.
$$
Since
$\{x^{b'} \colon x^{b'}<_{\rlex} \minSh(a)\} =
\{x^{b}x_i \colon x^{b}<_{\rlex} x^a,\ 1\leq i \leq n\}$,
the latter dimension equals
$$
\dim_K \bigcap_{x^{b'}<_{\rlex}\minSh(a),\ |b'|=d+1} \Kern_{S/I} (\tau_{f^{b'}})_{d+1}
$$
which is by Proposition
\ref{kern1}
equal to
$$
|\{x^c \in \Bc(S/\gin(I)) : |c|=d+1,\ \minSh(a) \leq_{\rlex} x^c\}|.
$$
Analogously we compute
$$
\dim_K \bigcap_{x^b\leq_{\rlex}x^a,\ |b|=d} \Kern_{S/I} (\tau_{f^b})_{d+1}
$$
$$
=|\{x^c \in \Bc(S/\gin(I)) : |c|=d+1,\ \maxSh(a) <_{\rlex} x^c\}|.
$$
Taking the difference in question gives rise to the desired formula.
\end{proof}

Given a graded ideal $I \subset S$ and a monomial $x^a \in S$
we define
\begin{eqnarray*}
d_I(a)
&=&
|\Sh(x^a/x_{\min(a)}) \cap \Bc(S/\gin(I))|\\
&=&
|\{x^{b} \in S :x^{b} \in \Bc(S/\gin(I)),\ |b|=|a| \text{ and } x^{b}/x_{\min(b)}=x^{a}/x_{\min(a)}\}|.
\end{eqnarray*}
Using $d_I(a)$ we can decide whether $x^a$ belongs to $\gin(I)$ or not.
Recall that a monomial ideal $I \subseteq S$ is called {\em strongly stable},
if $x_j x^a/x_i \in I$ for every $x^a \in I$ such that $x_i |x^a$ and $j \leq i$.
Since $\chara K=0$ it is well-known that $\gin(I)$ is strongly stable
(see, e.g., Eisenbud \cite{EI95}).

\begin{lem}
\label{helper1}
Let $I \subset S$ be a graded ideal and $x^a \in S_d$ be a monomial with $d>1$.
Then
$$
x^a \not\in \gin(I)
\Longleftrightarrow
\min(x^a/x_{\min(a)}) - \min(x^a) + 1 \leq d_I(a).
$$
\end{lem}
\begin{proof}
Let $x^{a'}=x^a/x_{\min(a)}$.
Then clearly $\min(a') \geq \min(a)$ and $x^a \in \Sh(x^{a'})$.
If $x^a \not\in \gin(I)$, then $x^b\not\in \gin(I)$ for $x^b=x_jx^{a'}$
and $\min(a) \leq j \leq \min(a')$ because $\gin(I)$ is strongly stable.
Hence
$\min(x^a/x_{\min(a)}) - \min(x^a) + 1 \leq d_I(a)$ because the considered monomials $x^b$
are elements of  $\Sh(x^{a'})$. Similarly we see that
$\min(x^a/x_{\min(a)}) - \min(x^a) + 1 \leq d_I(a)$ implies that
$x^a \not\in \gin(I)$.
\end{proof}

Since $\gin(\gin(I))=\gin(I)$ (see, e.g., \cite{CO04} for a proof)
it follows right from the definition that:

\begin{lem}
\label{easyhelper}
Let $I \subset S$ be a graded ideal and $x^a \in S_d$ be a monomial with $d>1$.
Then
$$
d_I(a)=d_{\gin(I)}(a).
$$
\end{lem}
%------------------------------------------------------------------------
%
%
%
%------------------------------------------------------------------------
\section{Fibre Products}
\label{fibre}
We recall some notation.
Let $\gin_{n}(I)$ be the generic initial ideal of a graded ideal
$I \subset K[x_1,\dots,x_n]$ with respect to rlex induced by $x_1>\dots>x_n$,
let $\gin_{m}(J)$ be the generic initial ideal of a graded ideal
$J \subset K[x_{n+1},\dots,x_{n+m}]$ with respect to rlex induced by $x_{n+1}>\dots>x_{n+m}$,
and let $\gin(L)$ be the generic initial ideal of a graded ideal
$L \subset S=K[x_1,\dots,x_{n+m}]$ with respect to rlex induced by
$x_1>\dots>x_n>x_{n+1}>\dots>x_{n+m}$.
Furthermore we define the ideals $\mm_n=(x_1,\dots,x_n) \subset K[x_1,\dots,x_n]$,
$\mm_m=(x_{n+1},\dots,x_{n+m}) \subset K[x_{n+1},\dots,x_{n+m}])$ and  $Q=\mm_n\mm_m \subset S$.
We prove first a special version of our main result:
\begin{prop}
\label{mainprop}
Let
$I \subset K[x_1,\dots,x_n]$ and
$J \subset K[x_{n+1},\dots,x_{n+m}]$
be graded ideals.
Then we have the following formulas for ideals in $S$:
\begin{enumerate}
\item
$
\gin( I+ \mm_m)=
\gin( \gin_{n}(I)+ \mm_m).
$
\item
$
\gin( J + \mm_n)=
\gin( \gin_{m}(J)+ \mm_n).
$
\end{enumerate}
\end{prop}
\begin{proof}
(i):
Let $f_1,\dots,f_{n+m} \subset S_1$ be a $K$-basis
such that for the $S$-automorphism
$\varphi \colon S \to S, x_i \mapsto f_i$
we can compute
$\gin( I+ \mm_m)$
and
$\gin( \gin_{n}(I)+ \mm_m)$
by using $\varphi^{-1}$.
We denote the residue class of  $f_i$   in $K[x_1,\dots,x_n]=S/\mm_m$
by $g_i$.

Let $\psi \colon K[x_1,\dots,x_n] \to K[x_1,\dots,x_n]$ be an $K[x_1,\dots,x_n]$-automorphism
such that the ideals
$\gin_{n}(I)$ and $\gin_{n}(\gin_{n}(I))$ can be computed by using $\psi^{-1}$.

For a monomial $x^a \in S$ we have that
\begin{eqnarray*}
&&
x^a\in \gin( I+ \mm_m)\\
&\Leftrightarrow&
\overline{f^a} \in \spanK_K \{\overline{f^b}: x^b<_{\rlex} x^a,\ |b|=|a|\} \subset S/(I+ \mm_m)\\
&\Leftrightarrow&
\overline{g^{a}} \in \spanK_K \{\overline{g^{b}}: x^b<_{\rlex} x^a,\ |b|=|a|\}
\subset K[x_1,\dots,x_n]/I\\
&\Leftrightarrow&
\overline{\psi^{-1}(g^{a})} \in \spanK_K \{\overline{\psi^{-1}(g^{b})}: x^b<_{\rlex} x^a,\ |b|=|a|\}
\subset K[x_1,\dots,x_n]/\psi^{-1}(I)\\
&\Leftrightarrow&
\overline{\psi^{-1}(g^{a})} \in \spanK_K \{\overline{\psi^{-1}(g^{b})}: x^b<_{\rlex} x^a,\ |b|=|a|\}
\subset K[x_1,\dots,x_n]/\gin_{n}(I).
\end{eqnarray*}
Here the first equality follows from Lemma \ref{computegin}.
The second one is trivial. The third one follows since $\psi^{-1}$ is an $K$-automorphism
and the last equality follows
since $S/\psi^{-1}(I)$
and $K[x_1,\dots,x_n]/\gin_{n}(I)=K[x_1,\dots,x_n]/\ini_{<_{\rlex}}(\psi^{-1}(I))$
share a common $K$-vector space basis.
Similarly we compute
\begin{eqnarray*}
&&
x^a\in \gin( \gin_{n}(I)+ \mm_m)\\
&\Leftrightarrow&
\overline{f^a} \in \spanK_K \{\overline{f^b}: x^b<_{\rlex} x^a,\ |b|=|a|\} \subset S/(\gin_{n}(I)+ \mm_m)\\
&\Leftrightarrow&
\overline{g^{a}} \in \spanK_K \{\overline{g^{b}}: x^b<_{\rlex} x^a,\ |b|=|a|\}
\subset K[x_1,\dots,x_n]/\gin_{n}(I)\\
&\Leftrightarrow&
\overline{\psi^{-1}(g^{a})} \in \spanK_K \{\overline{\psi^{-1}(g^{b})}: x^b<_{\rlex} x^a,\ |b|=|a|\}
\subset K[x_1,\dots,x_n]/\psi^{-1}(\gin_{n}(I))\\
&\Leftrightarrow&
\overline{\psi^{-1}(g^{a})} \in \spanK_K \{\overline{\psi^{-1}(g^{b})}: x^b<_{\rlex} x^a,\ |b|=|a|\}
\subset K[x_1,\dots,x_n]/\gin_{n}(\gin_{n}(I))\\
&\Leftrightarrow&
\overline{\psi^{-1}(g^{a})} \in \spanK_K \{\overline{\psi^{-1}(g^{b})}: x^b<_{\rlex} x^a,\ |b|=|a|\}
\subset K[x_1,\dots,x_n]/\gin_{n}(I).
\end{eqnarray*}
The last equality is new and it follows from the  fact that
$\gin_{n}(\gin_{n}(I))=\gin_{n}(I)$.
Thus we see that
$$
\gin( I+ \mm_m)=
\gin( \gin_{n}(I)+ \mm_m)
$$
and this shows (i).
Analogously one proves (ii).
\end{proof}

We need the following result which relates the numbers $d_I(a)$ introduced in Section \ref{linear}
of various ideals.
\begin{prop}
\label{mainhelper}
Let
$I \subset K[x_1,\dots,x_n]$ and
$J \subset K[x_{n+1},\dots,x_{n+m}]$
be graded ideals and $x^a \in S_{d}$ be a monomial with $d>1$.
We have that
$$
d_{I+J+Q}(a)=
d_{I+\mm_m}(a)+
d_{J+\mm_n}(a).
$$
\end{prop}
\begin{proof}
As a $K$-vector space we may see
$S/(I+\mm_m)$ and $S/(J+\mm_n)$ as subspaces of $S/(I+J+Q)$ using the standard $K$-bases
$\Bc(S/(I+\mm_m)),\Bc(S/(J+\mm_n))$ and $\Bc(S/(I+J+Q))$ respectively.
Using the equality $(I+\mm_m) \cap (J+\mm_n)=I+J+Q$ it follows from a direct computation
that the homomorphism of $K$-vector spaces
$$
\rho_d\colon
S/(I+\mm_m)_{d} \oplus S/(J+\mm_n)_{d}\to
 S/(I+J+Q)_{d},\
(g,h) \mapsto g+h,
$$
is surjective with trivial kernel for $d>0$. Thus it is an isomorphism for $d > 0$.

Let $x^{a'}=x^a/x_{\min(a)}$. Then $|a'|=d-1 > 0$.
The following  diagram
$$
\begin{CD}
S/(I+\mm_m)_{d} \oplus S/(J+\mm_n)_{d}
 @>\rho_{d} >>
S/(I+J+Q)_{d}
\\
@V\bigoplus_{x^b<_{\rlex}x^{a'}} \tau_{f^b} V\bigoplus_{x^b<_{\rlex}x^{a'}} \tau_{f^b}V
@V\bigoplus_{x^b<_{\rlex}x^{a'}} \tau_{f^b}VV\\
\bigoplus_{x^b<_{\rlex}x^{a'}} S/(I+\mm_m)_{1} \oplus \bigoplus_{x^b<_{\rlex}x^{a'}} S/(J+\mm_n)_{1}
@>\oplus_{x^b<_{\rlex}x^{a'}} \rho_{1}>>
\bigoplus_{x^b<_{\rlex}x^{a'}} S/(I+J+Q)_{1}
\end{CD}
$$
is commutative where the direct sums are take over monomials of degree $d-1$.
It follows that
$$
\dim_K \bigcap_{x^b<_{\rlex}x^{a'},\ |b|=d-1} \Kern_{S/(I+J+Q)} (\tau_{f^b})_{d}
$$
$$=
\dim_K \bigcap_{x^b<_{\rlex}x^{a'},\ |b|=d-1} \Kern_{S/(I+\mm_m)} (\tau_{f^b})_{d}
+
\dim_K \bigcap_{x^b<_{\rlex}x^{a'},\ |b|=d-1} \Kern_{S/(J+\mm_m)} (\tau_{f^b})_{d}
$$
Analogously we see that
$$
\dim_K \bigcap_{x^b\leq_{\rlex}x^{a'},\ |b|=d-1} \Kern_{S/(I+J+Q)} (\tau_{f^b})_{d}
$$
$$=
\dim_K \bigcap_{x^b\leq_{\rlex}x^{a'},\ |b|=d-1} \Kern_{S/(I+\mm_m)} (\tau_{f^b})_{d}+
\dim_K \bigcap_{x^b\leq_{\rlex}x^{a'},\ |b|=d-1} \Kern_{S/(J+\mm_m)} (\tau_{f^b})_{d}
$$
Now Proposition \ref{ker2} implies that
\begin{eqnarray*}
d_{I+J+Q}(a) &=& |\Sh(a') \cap \Bc(S/\gin(I+J+Q))|\\
&=&
|\Sh(a') \cap \Bc(S/\gin(I+\mm_m))|+ |\Sh(a') \cap \Bc(S/\gin(J+\mm_n))|\\
&=&
d_{I+\mm_m}(a)+
d_{J+\mm_n}(a).
\end{eqnarray*}
This concludes the proof.
\end{proof}

We are ready to prove Theorem \ref{Nevo1}.

\begin{proof}[Proof of \ref{Nevo1}]
Let $x^a \in S_d$ be a monomial for an integer $d>0$.
If $d=1$ then
$\gin( I+ J + Q)_1=
\gin( \gin_{n}(I)+ \gin_{m}(J) + Q)_1$ for trivial reasons.
Assume that $d>1$.
Applying Lemma \ref{easyhelper},
Proposition \ref{mainprop} and  Proposition \ref{mainhelper} we see that
\begin{eqnarray*}
d_{I+J+Q}(a)
&=&
d_{I+\mm_m}(a) + d_{J+\mm_n}(a)
=
d_{\gin(I+\mm_m)}(a) + d_{\gin(J+\mm_n)}(a)\\
&=&
d_{\gin(\gin_{n}(I)+\mm_m)}(a) + d_{\gin(\gin_{m}(J)+\mm_n)}(a)
=
d_{\gin_{n}(I)+\mm_m}(a) + d_{\gin_{m}(J)+\mm_n}(a)\\
&=&
d_{\gin_{n}(I)+\gin_{m}(J)+Q}(a).
\end{eqnarray*}
Thus it follows from Lemma \ref{helper1} that
$$
x^a \not\in \gin( I+ J + Q)
\Leftrightarrow
x^a \not\in \gin( \gin_{n}(I)+ \gin_{m}(J) + Q).
$$
\end{proof}

\begin{rem} It is clear that rlex plays an important role in the arguments we have used to prove \ref{Nevo1}.
For other term orders  the statement of \ref{Nevo1} fails to be true.
For instance, for the lex order, $m=n=3$ the ideals  $I=(x_1^2, x_1x_2)$ and $J=(x_4^2, x_5x_6)$  do not fulfill \ref{Nevo1}.
 \end{rem}

In view of Theorem \ref{Nevo1}, the ideal  $\gin( F(I,J))$ can be computed by first computing   $\gin_n(I),\gin_m(J)$
and then  $\gin( F(\gin_n(I),\gin_m(J)))$.  One would like to describe the generators of
the ideal  $\gin( F(I,J))$ in terms of those of $\gin_n(I)$ and $\gin_m(J)$. The first step in this direction is the following:

 \begin{prop}
\label{computegin2}
Let $Q=(x_i : 1\leq i\leq n)(x_j : n+1\leq j\leq m)\subset S$.
Then:
\begin{enumerate}
\item $\gin(Q)=
(x_ix_j: i+j \leq n+m,\ i \leq j \text{ \rm and } i \leq \min\{n,m\})$.
\item  $\gin(Q^k)=\gin(Q)^k$ for all $k$.
\end{enumerate}
\end{prop}

\begin{proof}
Every ideal which is a product of ideals of linear forms has a linear resolution, see \cite{COHE03}.
The gin-rlex  of an ideal with linear resolution has linear resolution \cite{EI95}.
Hence we know that $\gin(Q^k)$ is generated in degree $2k$ for all $k$.
Therefore it is enough to prove equality  (i) in degree $2$ and to prove equality (ii) in degree $2k$.
To check that (i) holds in degree $2$, it is enough to prove the inclusion $\supseteq$
since the two vector spaces involved have  vector space dimension equal to $mn$.
It follows from Lemma \ref{helper1}
that for a monomial $x_ix_j$ with $i\leq j$ we have
$$
x_ix_j \in \gin(Q)
\Longleftrightarrow
j - i + 1 > d_Q(\epsilon_i+\epsilon_j)
$$
where for $k\in \{1,\dots,n+m\}$ we denote by $\epsilon_k$ the $k$-th
standard basis vector of $\Z^{n+m}$.
By Lemma \ref{easyhelper} and Proposition \ref{mainhelper}
we know
\begin{eqnarray}
\label{eqlast}
d_{Q}(\epsilon_i+\epsilon_j)
&=&
d_{\mm_m}(\epsilon_i+\epsilon_j)+
d_{\mm_n}(\epsilon_i+\epsilon_j)=
d_{\gin(\mm_m)}(\epsilon_i+\epsilon_j)+
d_{\gin(\mm_n)}(\epsilon_i+\epsilon_j)\\
&=&
d_{(x_1,\dots,x_m)}(\epsilon_i+\epsilon_j)+
d_{(x_1,\dots,x_n)}(\epsilon_i+\epsilon_j)
\end{eqnarray}
where we used the obvious facts
$\gin(\mm_m)=(x_1,\dots,x_m)$
and $\gin(\mm_n)=(x_1,\dots,x_n)$.

Thus we have to compute
$d_{(x_1,\dots,x_m)}(\epsilon_i+\epsilon_j)$
and
$d_{(x_1,\dots,x_n)}(\epsilon_i+\epsilon_j)$ respectively.
 We treat the case $n \leq m$; if $n>m$ one argues  analogously.
At first we consider
$$
1\leq i,j \leq n+m \text{ with }
i+j \leq n+m,\ i \leq j \text{ and } i \leq n.
$$
It follows from the definition that
$$
d_{(x_1,\dots,x_m)}(\epsilon_i+\epsilon_j)=
|\{ x_lx_j : m+1 \leq l,j \leq m+n,\ l\leq j \} |.
$$
This number is only nonzero if  $j \geq m+1$ and in this case we have
$
d_{(x_1,\dots,x_m)}(\epsilon_i+\epsilon_j)=
j-m.
$
Similarly, we have that
$
d_{(x_1,\dots,x_n)}(\epsilon_i+\epsilon_j)
$
is only nonzero if $j \geq n+1$ and in this case we have
$
d_{(x_1,\dots,x_n)}(\epsilon_i+\epsilon_j)=
j-n.
$
In any case it follow from Equation (\ref{eqlast}) that
$$
d_{Q}(\epsilon_i+\epsilon_j)
\leq
(j-n) + (j-m)
\leq
(m-i) + (j-m)<
j-i+1
$$
where we used that $i+j \leq m+n$.
Hence we get that $x_ix_j \in \gin(Q)$ as desired. As explained above, this proves (i).

As observed above, to prove (ii) it is enough to prove the asserted equality in degree $2k$.
Consider the $K$-algebra $K[\varphi(Q_2)]$ generated by $\varphi(Q_2)$ for a generic change of coordinates $\varphi$.
The assertion we have to prove is equivalent to the assertion that the initial algebra
of $K[\varphi(Q_2)]$ is generated as a $K$-algebra by the elements $x_ix_j$
described in (i) which generate $\gin(Q)$ as an ideal.
We claim that the algebra relations among those $x_ix_j$  are of degree $2$.
This  can be easily seen  by eliminating from the
second Veronese ring of $S$ using the know Gr\"obner bases \cite{CO94} or \cite{ST96}.
Hence,  by the Buchberger-like criterion
for being a Sagbi basis \cite{COHEVA96},
it suffices to prove that $\gin(Q^2)=\gin(Q)^2$ in degree $4$.
This can be done by checking that the two vector spaces involved have the same dimension.
The dimension of $\gin(Q^2)_4$ is equal to that of $Q^2_4$ which is $\binom{n+1}{2}\binom{m+1}{2}$.
It remains to compute the dimension of  $\gin(Q)^2_4$. Observe that a basis of   $\gin(Q)^2_4$
is given by the monomials $x_ix_jx_hx_k$ such that the following conditions hold:
\begin{equation}\label{ijhk}
\left\{ \begin{array}{l}
1\leq i\leq j\leq h\leq k\leq m+n, \\
 j\leq \min(n,m),\\
 i+k\leq n+m, \\
 j+h\leq n+m.
 \end{array}
 \right.
 \end{equation}
It would be nice to have a bijection showing that the number, call it  $W(n,m)$,  of $(i,j,h,k)$  satisfying (\ref{ijhk}) is equal to  $\binom{n+1}{2}\binom{m+1}{2}$.  We content ourself with a non-bijective proof.  Assuming  that $n\leq m$, a simple  evaluation of the formula
$$\sum_{i=1}^n\  \sum_{j=i}^n\  \sum_{h=j}^{m+n-j}\ \sum_{k=h}^{m+n-i}1$$
 allows us to write $W(n,m)$ as
\begin{equation} \label{part1}
\sum_{i=1}^n  \binom{n+m-2i+3}{ 3}
\end{equation}
plus
\begin{equation}\label{part2}
-\binom{m+2}{4}-\binom{n+2}{4}+\binom{m-n+2}{4}.
\end{equation}
The expression (\ref{part1}) can be evaluated, depending on whether $m+n$ is odd or even,  using the formulas:
\begin{equation}
\sum_{i=0}^p \binom{2i+u}{3}=\left\{
\begin{array}{ll}
p(p+1)(2p^2+2p-1)/6 & \mbox{if } u=1,\\
(p^2-1)p^2/3 & \mbox{if } u=0.
\end{array}\right.
\end{equation}
It follows that the expression of (\ref{part1}) is indeed equal to
\begin{equation} \label{part3}
1/6n^3m + 1/6nm^3 + 1/6n^3 + 1/2nm^2 + 1/6nm - 1/6n
\end{equation}
That summing (\ref{part3}) with  (\ref{part2}) one gets $\binom{m+1}{2}\binom{n+1}{2}$ is a straightforward computation.
\end{proof}

\begin{rem}
\begin{enumerate}
\item[(a)]  Refining the  arguments \ref{computegin2}  one can show that statement (i) and (ii)
hold  for every term order satisfying $x_1>\dots>x_{m+n}$.
\item[(b)] The equality $\gin(Q^k)=\gin(Q)^k$ is quite remarkable. We do not know other families  of ideals satisfying it. It is difficult to guess a generalization of it. For instance that equality  does not hold for a product of $3$ or more ideals of linear forms, e.g. $(x_1,x_2)(x_3,x_4)(x_5, x_6)$. And it does not hold for the ideal defining the fibre product of $3$ of more polynomial rings, e.g.  $(x_1,x_2)(x_3,x_4)+(x_1,x_2)(x_5, x_6)+(x_3,x_4)(x_5, x_6)$.
\item[(c)] The statements (i) and (ii) of  \ref{computegin2}, can be rephrased as follows: The $(n,m)$-th  Segre product  in generic coordinates has a toric degeneration to a ladder of the second Veronese ring  of the polynomial ring in  $m+n-1$ variables  and this holds independently of  the term order.
\end{enumerate}
\end{rem}

%------------------------------------------------------------------------
%
%
%
%------------------------------------------------------------------------
\section{Symmetric shifted simplicial complexes}
\label{symmetric}

Recall that $\Gamma$ is called a {\em simplicial complex} on the vertex set
$\{x_1,\dots,x_n\}$ if $\Gamma$ is a subset of the power set of $\{x_1,\dots,x_n\}$
which is closed under inclusion.
The elements $F$ of $\Gamma$ are called {\em faces}, and the maximal
elements under inclusion are called {\em facets}.
If $F$ contains $d+1$ elements of $[n]$, then $F$ is called a {\em $d$-dimensional} face, and we write $\dim F = d$.
The empty set is a face of dimension $-1$.
If $\Gamma \neq \emptyset$,
then the {\em dimension} $\dim \Gamma $ is the maximum of the dimensions of the faces of $\Gamma$.
Let $f_i$ be the total number of $i$-dimensional faces of $\Gamma$
for $i=-1,\dots,\dim \Gamma$. The vector
$f(\Gamma)=(f_{-1},\dots,f_{\dim \Gamma})$ is called the {\em $f$-vector} of $\Gamma$.

Let $\mathcal{C}(x_1,\dots,x_n)$ be the set of simplicial complexes on $\{x_1,\dots,x_n\}$.
Following constructions of Kalai we define axiomatically the concept of algebraic shifting.
See  \cite{HE01} and \cite{KA01} for surveys on this subject.
We call a map $\Delta \colon \mathcal{C}(x_1,\dots,x_n) \to \mathcal{C}(x_1,\dots,x_n)$
an {\em algebraic shifting operation} if the following
is satisfied:

\begin{enumerate}
\item[(S1)]
If $\Gamma \in \mathcal{C}(x_1,\dots,x_n)$,
then $\Delta(\Gamma)$ is a {\em shifted complex},
\ie for all $F \in \Delta(\Gamma)$ and $i<j$ with $x_i\in F$
we have that $F\setminus\{x_i\} \cup \{x_j\} \in \Delta(\Gamma)$.
\item[(S2)]
If $\Gamma \in \mathcal{C}(x_1,\dots,x_n)$ is shifted, then $\Delta(\Gamma)=\Gamma$.
\item[(S3)]
If $\Gamma \in \mathcal{C}(x_1,\dots,x_n)$, then $f(\Gamma)=f(\Delta(\Gamma))$.
\item[(S4)]
If $\Gamma', \Gamma \in \mathcal{C}(x_1,\dots,x_n)$
and $\Gamma' \subseteq \Gamma$ is a subcomplex,
then $\Delta(\Gamma') \subseteq \Delta(\Gamma)$ is a subcomplex.
\end{enumerate}

Let $K$ be a field and $S=K[x_1,\dots,x_n]$ be the polynomial ring
in $n$ indeterminates.
We define the {\em Stanley-Reisner ideal}
$I_\Gamma=(\prod_{x_i \in F} x_i : F\subseteq \{x_1,\dots,x_n\},\ F\not\in \Gamma)$
and
the {\em Stanley-Reisner ring}
$K[\Gamma]=S/I_\Gamma$.
Denote by $x_F$ the squarefree monomial $\prod_{x_i \in F} x_i$.

Note that a simplicial complex $\Gamma$ is shifted
if and only if the Stanley-Reisner ideal
$I_\Gamma \subset S=K[x_1,\dots,x_n]$
is a {\em squarefree strongly stable ideal} with respect to $x_1>\dots>x_n$,
\ie for all $F\subseteq \{x_1,\dots,x_n\}$ with $x_F \in I_\Gamma$
and all $x_i \in F$, $j<i$, $x_j \not\in F$
we have that $(x_jx_F)/x_i \in I_\Gamma$.

Here we are interested in the following shifting operation introduced by Kalai.
We follow the algebraic approach of Aramova, Herzog and Hibi.
Assume that $K$ is a field of characteristic $0$.
We consider the following operation on monomial ideals of $S$.
For a monomial
$m=x_{i_1}\cdots x_{i_t}$ with $i_1\leq i_2 \dots\leq i_t$
of $S$ we set $m^\sigma=x_{i_1}x_{i_2+1}\dots x_{i_t+ t-1}$.
For a monomial ideal $I$ with unique minimal system of generators $G(I)=\{m_1,\dots,m_s\}$
we set $I^\sigma=(m_1^\sigma,\dots,m_s^\sigma)$ in a suitable polynomial ring with sufficiently many variables.

Let $\Gamma$ be a simplicial complex on the vertex set $\{x_1,\dots,x_n\}$ with
Stanley-Reisner
ideal
$I_{\Gamma}$.
The {\em symmetric algebraic shifted complex} of $\Gamma$ is the unique
simplicial complex $\Delta^s(\Gamma)$ on the vertex set $\{x_1,\dots,x_n\}$ such that
$
I_{\Delta^s(\Gamma)} = \gin(I_\Gamma)^\sigma \subset S
$
where
$\gin(\cdot)$ is
the generic initial ideal  with respect to the reverse lexicographic order induced by $x_1>\dots>x_n$.
It is not obvious that $\Delta^s(\cdot)$ is indeed a shifting operation.
The first difficulty is already that
$I_{\Delta^s(\Gamma)}$ is an ideal of $S$. This and the proofs of the other properties can be found in \cite{HE01}.

We want to give a combinatorial application of the results of Section \ref{fibre}.
For two simplicial complexes
$\Gamma_1$ on the vertex set $\{x_1,\dots,x_n\}$ and
$\Gamma_2$ on the vertex set $\{x_{n+1},\dots,x_{n+m}\}$
we denote by
$$\Gamma_1 \uplus \Gamma_2$$
the {\em disjoint union}
of $\Gamma_1$ and $\Gamma_2$ on the vertex set $\{x_1,\dots,x_n,x_{n+1},\dots,x_{n+m}\}$.
The following theorem was conjectured by Kalai in \cite{KA01}.
Observe that it was proved for exterior shifting by Nevo \cite{NE03}.

\begin{thm}
\label{mainshifted}
Let $\Gamma_1$ be a simplicial complex on the vertex set $\{x_1,\dots,x_n\}$ and
$\Gamma_2$ be a simplicial complex on the vertex set
$\{x_{n+1},\dots,x_{n+m}\}$. Then we have that
$$
\Delta^s(\Gamma_1 \uplus \Gamma_2)=
\Delta^s(\Delta^s(\Gamma_1) \uplus \Delta^s(\Gamma_2)).
$$
\end{thm}
\begin{proof}
The Stanley-Reisner ideal of $\Gamma_1 \uplus \Gamma_2$
is
$
I_{\Gamma_1 \uplus \Gamma_2}=I_{\Gamma_1}+ I_{\Gamma_2} + Q.
$
To determine the simplicial complex  $\Delta^s(\Gamma_1 \uplus \Gamma_2)$
we compute
\begin{eqnarray*}
I_{\Delta^s(\Gamma_1 \uplus \Gamma_2)}
&=&
\gin( I_{\Gamma_1}+ I_{\Gamma_2} + Q)^\sigma\\
&=&
\gin( \gin_n(I_{\Gamma_1})+ \gin_m(I_{\Gamma_2}) + Q)^\sigma\\
&=&
\gin( \gin_n(\gin_n(I_{\Gamma_1})^{\sigma_n})+ \gin_m(\gin_m(I_{\Gamma_2})^{\sigma_m}) + Q)^\sigma\\
&=&
\gin( \gin_n(I_{\Gamma_1})^{\sigma_n}+ \gin_m(I_{\Gamma_2})^{\sigma_m} + Q)^\sigma
=
I_{\Delta^s(\Delta^s(\Gamma_1) \uplus \Delta^s(\Gamma_2))}.
\end{eqnarray*}
The first and fifth equality are the definitions of
the corresponding ideals.
The second and forth equality follow from \ref{Nevo1}.
The third equality follows from the proof in \cite[Theorem 8.19]{HE01}.
\end{proof}

%------------------------------------------------------------------------
%
%
%
%------------------------------------------------------------------------
\section{Algebraic properties of fibre products}
\label{properties}

It is easy to see how homological properties transfer from $A$ and $B$ to $A\circ B$.
The main point is that the short exact Sequence (\ref{strfibre})  allows to compare how to control local cohomology.
In this section we show that the componentwise linearity behaves well with respect to fibre products.
Recall the following definition of Herzog-Hibi \cite{HEHI99}.
A graded ideal $L$ in some polynomial ring over the field $K$ is called {\em componentwise linear}
if for all $k\in \N$ the ideal
$L_{\langle k \rangle}$ has a $k$-linear resolution (i.e. $\Tor_i(L_{\langle k \rangle},K)_{i+j}=0$ for $j\neq k$)
where $L_{\langle k \rangle}$ is the ideal generated by all elements of degree $k$ of $L$.

In the proof we will use the following criterion of Herzog-Reiner-Welker \cite[Theorem 17]{HEREWE99}
to check whether an ideal is componentwise linear:
Let $L$ be a graded ideal in a polynomial ring $T$.
Then $L$ is componentwise linear
if and only if for all $k \in \Z$ we have $\reg T/L_{\leq k} \leq k-1$ where
$L_{\leq k}$ is the ideal generated by all elements of $L$
of degree smaller or equal to $k$ and $\reg M$ denotes the Castelnuovo-Mumford regularity of a
f.g.\ graded $T$-module $M$.
Note that in \cite{HEREWE99} this result is only stated for monomial ideals,
but the proof works also, more generally, for graded ideals.
Recall that the Castelnuovo-Mumford regularity can be computed via local cohomology as
$
\reg M=\max\{ j \in \Z : H^i_\mm(M)_{j-i} \neq 0 \text{ for some } i\}$
where $\mm$ denotes also the graded maximal ideal of $T$.
\begin{thm}
Let
$I \subset K[x_1,\dots,x_n]$ and
$J \subset K[x_{n+1},\dots,x_{n+m}]$
be graded ideals such that $I \subseteq \mm_n^2$ and $J \subseteq \mm_m^2$.
Then
$I$ and $J$ are componentwise linear if and only if $F(I,J)$ is componentwise linear.
\end{thm}
\begin{proof}
At first we note that $F(I,J)_{\leq k}=0$
for $k \leq 1$,
and for $k \geq 2$  we have
$$
F(I,J)_{\leq k}
=
(I+J+Q)_{\leq k}
=
I_{\leq k}+J_{\leq k}+Q
=
N(I_{\leq k},J_{\leq k}).
$$
The short exact Sequence (\ref{strfibre}) with the current notation  is:
$$
0
\to
S/F(I,J)_{\leq k}
\to
S/(I_{\leq k} + \mm_m)
\oplus
S/(J_{\leq k} + \mm_n)
\to
S/\mm
\to
0.
$$
It induces the long exact sequence of local cohomology modules
in degree $j-i$
$$
\dots \to
H_\mm^{i-1} (S/\mm)_{j-1-(i-1)}
\to
H_\mm^{i}( S/F(I,J)_{\leq k})_{j-i}
$$
$$
\to
H_\mm^{i}(S/(I_{\leq k} + \mm_m))_{j-i}
\oplus
H_\mm^{i}(S/(J_{\leq k} + \mm_n))_{j-i}
\to
H_\mm^{i}(S/\mm)_{j-i}
\to \dots
$$
Observe that for $j\geq 2$ we have that
$H_\mm^{i-1} (S/\mm)_{j-1-(i-1)}=H_\mm^{i}(S/\mm)_{j-i}=0$ because $S/\mm$ has regularity $0$.
The ring change isomorphism for local cohomology modules yields
$$
H_\mm^{i}(S/(I_{\leq k} + \mm_m))\cong H_{\mm_n}^{i}(K[x_1,\dots,x_n]/I_{\leq k})
$$
and
$$
H_\mm^{i}(S/(J_{\leq k} + \mm_n))\cong H_{\mm_m}^{i}(K[x_{n+1},\dots,x_{n+m}]/J_{\leq k})
$$
as graded $K$-vector spaces.
Thus we get for $j \geq 2$ the isomorphisms
$$
H_\mm^{i}( S/F(I,J)_{\leq k})_{j-i}
\cong
H_{\mm_n}^{i}(K[x_1,\dots,x_n]/I_{\leq k})_{j-i}
\oplus
H_{\mm_m}^{i}(K[x_{n+1},\dots,x_{n+m}]/J_{\leq k})_{j-i}
$$
as $K$-vector spaces. Since the regularities of all modules in questions are greater or equal to 2,
we obtain that
$\reg S/F(I,J)_{\leq k} \leq k$ if and only if
$\reg K[x_1,\dots,x_n]/I_{\leq k} \leq k$
and
$\reg K[x_{n+1},\dots,x_{n+m}]/J_{\leq k} \leq k$.
As noted above, this concludes the proof.
\end{proof}

\end{document}